\documentclass[a4paper]{amsart}

\usepackage[british]{babel}

\usepackage{mlmodern}
\DeclareFontFamily{OMX}{mlmex}{}
\DeclareFontShape{OMX}{mlmex}{m}{n}{<->mlmex10}{}
\usepackage[T1]{fontenc}
\usepackage[babel,stretch=10,shrink=10,verbose=errors,selected]{microtype}
\usepackage[strict,english=british]{csquotes}

\usepackage{amsmath,amssymb,amsthm}
\usepackage{mathtools}

\usepackage[
	backend=biber,
	style=numeric-comp,
	sorting=nyt,
	sortcites=true,
	maxnames=5,
	giveninits=true,
	date=year,
	useprefix=true,
]{biblatex}
\DeclareFieldFormat{eid}{no.~#1}
\usepackage{enumitem}

\usepackage{xcolor}
\usepackage[nobiblatex]{xurl}
\usepackage[breaklinks=true,pdflang=en-GB,colorlinks=true]{hyperref}
\colorlet{citecolor}{green!75!black}
\colorlet{linkcolor}{red!75!black}
\colorlet{urlcolor}{blue!75!black}
\hypersetup{citecolor=citecolor,linkcolor=linkcolor,urlcolor=urlcolor}
\usepackage{zref-clever}
\zcsetup{cap,nameinlink=false}
\usepackage{keytheorems}

\usepackage{orcidlink}
\usepackage{ellipsis}

\newkeytheorem{theorem}[parent=section,style=plain]
\newkeytheorem{proposition}[sibling=theorem,style=plain]
\newkeytheorem{corollary}[sibling=theorem,style=plain]

\newkeytheorem{remark}[sibling=theorem,style=remark]

\newcommand{\set}[2]{\ensuremath{\{\,{#1}\mid{#2}\,\}}}
\newcommand{\ord}[1]{\ensuremath{\left|{#1}\right|}}
\newcommand{\ssetminus}{\!\setminus\!}

\newcommand{\NN}{\mathbb{N}}
\newcommand{\RR}{\mathbb{R}}

\DeclareMathOperator{\GL}{GL}
\DeclareMathOperator{\PGL}{PGL}
\DeclareMathOperator{\SL}{SL}
\DeclareMathOperator{\PSL}{PSL}

\title{Class numbers and nilpotent subgroups of \(\PGL(2,q)\)}
\author[S. Tertooy]{Sam Tertooy\ \orcidlink{0000-0002-5750-9153}}
\date{\today}
\address{KU Leuven, Kulak Kortrijk Campus\\
	E.~Sabbelaan 53\\
	8500 Kortrijk\\
	Belgium}
\email{\href{mailto:sam.tertooy@kuleuven.be}{sam.tertooy@kuleuven.be}}
\urladdr{\url{https://stertooy.github.io}}


\begin{document}

\begin{abstract}
We show that for certain odd prime powers \(q\), the number of conjugacy classes of \(\PGL(2,q)\) is greater than the order of its largest nilpotent subgroup. This answers negatively a question of Liebeck and Pyber.
\end{abstract}

\maketitle

\section{Introduction}

For a finite group \(G\), let \(k(G)\) denote its class number (i.e.\@ the number of conjugacy classes) and let \(n(G)\) denote the order of its largest nilpotent subgroup. In \cite{lp97-a}, Liebeck and Pyber prove the existence of a constant \(c < 58/21\) such that the inequality \(k(G) \leq n(G)^c\) holds for every finite group \(G\), and ask whether this holds with \(c = 1\). This is listed as question 14.54 in the Kourovka Notebook \cite{km26-a}.

There are many examples of finite groups for which \(k(G) = n(G)\) holds, including all finite abelian groups, all holomorphs \(C_p \rtimes C_{p-1}\) for \(p\) prime, and all projective special linear groups \(\PSL(2,q)\) with \(q\) even. By analysing their nilpotent subgroups, we prove that for \(G \coloneq \PGL(2,q)\) with \(q \neq 2^r + \epsilon\) for any \(r \in \NN\) and \(\epsilon \in \{-1,0,1\}\), \(k(G) = q+2 > q+1 = n(G)\). As a consequence, \(c \geq \log_{12} 13 > 1\). More examples of groups with \(k(G) > n(G)\) can then be constructed using these projective linear groups.

\section{\texorpdfstring{Nilpotent subgroups of \(\PGL(2,q)\)}{Nilpotent subgroups of PGL(2,q)}}

In this section, we determine \(n(\PGL(2,q))\)  for all prime powers \(q\). The subgroups of \(\PSL(2,q)\) have been classified in e.g.\@ \cite[Thm.~2.1]{king05-a}. For \(q\) even, \(\PGL(2,q)\) equals \(\PSL(2,q)\), so we already know its subgroups. For \(q\) odd, \(\PGL(2,q)\) contains \(\PSL(2,q)\) as an index 2 subgroup, and the former's subgroups can be obtained from the latter's. The resulting classification can be found in e.g.\@ \cite[Thm.~2]{cot06-a}. We summarise these results below.

\begin{theorem}
	Let \(q = p^n\) for \(p\) prime. The subgroups of \(\PGL(2,q)\) are:
\begin{enumerate}[label=(\roman*)]
\item Abelian groups:
\begin{enumerate}[label=(\alph*)]
\item \(C_2\),
\item \(C_d\), for every \(d > 2\) with \(d \mid q \pm 1\),
\item \((C_2)^2\), unless \(q = 2\),
\item \((C_p)^m\), for every \(m \leq n\).
\end{enumerate} 
\item Non-abelian dihedral groups:
\begin{enumerate}[label=(\alph*)]
\item \(D_{2p}\) (when \(p > 2\)),
\item \(D_{2d}\), for every \(d > 2\) with \(2d \mid q \pm 1\),
\item \(D_{2d}\), for every \(d > 2\) with \(d \mid q \pm 1\) and \((q \pm 1)/d\) odd.
\end{enumerate}
\item Certain non-abelian symmetric and alternating groups.
\item \(\PSL(2,p^m)\) and \(\PGL(2,p^m)\), for every \(m \mid n\).

\item Non-abelian semi-direct products \((C_p)^m \rtimes C_d\) for every \(m \leq n\) and \(d \geq 2\) with \(d \mid p^{\gcd(m,n)}- 1\).
\end{enumerate}
\end{theorem}

Using this classification, we can state the following \zcref[nocap,noref]{prop:largestnilpPGL}.

\begin{proposition}
	\label{prop:largestnilpPGL}
	The largest nilpotent subgroup of \(G \coloneq \PGL(2,q)\) has order
	\[
	n(G) = \begin{cases}
		q + 1 &\text{ if } q \neq 2^r \pm 1, \\
		2^{r+1} &\text{ if } q = 2^r \pm 1,
	\end{cases}
	\]
	where \(r \geq 2\).
\end{proposition}
\begin{proof}
Family (i) consists entirely of abelian groups. The largest group in this family is \(C_{q+1}\), and in the case \(q = 3\) it is joint largest with \((C_2)^2\). In any case, the largest order is always \(q+1\).

For family (ii), recall that a dihedral group is nilpotent if and only if its order is a power of \(2\). Thus, the largest nilpotent dihedral subgroup has order \(2d\), where \(d\) is the largest power of \(2\) dividing \(q \pm 1\). If \(q = 2^r \pm 1\) this subgroup has order \(2(q\mp1) = 2^{r+1}\), otherwise, it has order at most \(\frac{2}{3}(q+1)\).

Families (iii), (iv) and (v) never contain nilpotent groups. For families (iii) and (iv) this is easily checked. For family (v), recall that a finite group is nilpotent if and only if it is a direct product of its Sylow subgroups. But here \((C_p)^m\) and \(C_d\) do not commute.
\end{proof}

\section{\texorpdfstring{Groups with \(k(G) > n(G)\)}{Groups with k(G) > n(G)}}

In \cite[Tbl.~3]{macd81-a}, Macdonald shows that for \(G = \PGL(2,q)\), one has \(k(G) = q + \gcd(2,q-1)\). Combining this with \zcref{prop:largestnilpPGL} we may conclude the following:
\begin{theorem}
\label{thm:mainthm}
Let \(G = \PGL(2,q)\).
\begin{enumerate}
\item If \(q = 2^r\) with \(r \geq 1\), then \(k(G) = q+1 = n(G)\). 
\item If \(q = 2^r \pm 1\) with \(r \geq 2\), then \(k(G) = q+2  < 2^{r+1} = n(G)\).
\item If \(q \neq 2^r \pm \epsilon\) with \(\epsilon \in \{-1,0,1\}\), then \(k(G) = q+2  > q+1 = n(G)\).
\end{enumerate}
Thus \(k(G) > n(G)\) if and only if \(q \neq 2^r + \epsilon\) for any \(r \in \NN\) and \(\epsilon \in \{-1,0,1\}\).
\end{theorem}

\begin{corollary}
Let \(c \in \RR \) be such that for every finite group \(G\), \(k(G) \leq n(G)^c\). Then \(c \geq \log_{12} 13 > 1\).
\end{corollary}

The two smallest groups for which \(k(G) > n(G)\) are \(\PGL(2,11)\) (order \(1320\), class number \(13\)) and \(\PGL(2,13)\) (order \(2184\), class number \(15\)).
Using the \textsc{GAP} \cite{gap26-a} packages \textsc{SmallGrp} \cite{beo24-a}, \textsc{GrpConst} \cite{be24-a} and \textsc{SmallClassNr} \cite{tert26-c}, we verified that no other examples exist among the groups of order \(\ord{G} \leq 2303\), nor among the groups of class number \(k(G) \leq 14\).


More examples can be created from these projective linear groups. A first way to do so is by using direct products, since both \(k(.)\) and \(n(.)\) behave multiplicatively in this case. Taking the direct product of any finite group with sufficiently many copies of \(\PGL(2,q)\) (\(q \neq 2^r \pm \epsilon\)) will result in a group \(G\) satisfying \(k(G) > n(G)\). For example, \(H \coloneq \SL(2,3)\) has \(k(H) = 7 < 8 = n(H)\), but the direct product \(G \coloneq H \times \left(\PGL(2,11)\right)^2\) has \(k(G) = 1183 > 1152 = n(G)\).

A second way to construct new examples is by taking subdirect products, making use of the fact that \(\PGL(2,q)\) with \(q\) odd contains \(\PSL(2,q)\) as an index \(2\) (hence normal) subgroup.

\begin{corollary}
Let \(G\) be a subdirect product \(P \times_{C_2} A\) where \(P \coloneq \PGL(2,q)\) with \(q\) odd and \(A\) is an abelian group of even order. Then \(k(G) > n(G)\) if and only if \(q \neq 2^r \pm 1\).
\end{corollary}
\begin{proof}
Let \(\pi_1\colon P \to C_2\) and \(\pi_2\colon A \to C_2\) be the epimorphisms such that
\[ G = \set{(x,y) \in P \times A}{\pi_1(x) = \pi_2(y)},\]
and set \(K_i \coloneq \ker(\pi_i)\). If \(N\) is the largest nilpotent subgroup in \(\PGL(2,q)\), then \(N \times A\) is the largest nilpotent subgroup of \(P \times A\) and \(G \cap (N \times A)\) is the largest nilpotent subgroup of \(G\). Setting \(M \coloneq N \cap K_1\), we see that
\[G \cap (N \times A) = (M \times K_2) \sqcup (N \ssetminus M \times A \ssetminus K_2),\]
so this subgroup has order \(\ord{N}\!\ord{A}/2\).

For any subset \(X \subseteq P\), let \(k_P(X)\) denote the number of \(P\)-conjugacy classes contained in \(X\). Since elements of \(G\) are conjugate in \(G\) if and only if they are conjugate in \(P \times A\), we can similarly conclude that
\[ k(G) = k_{P \times A}(G) = k_P(K_1) \ord{K_2} + k_P(P \ssetminus K_1) \ord{A \ssetminus K_2} = k(P) \ord{A} / 2.\]
In conclusion, \(n(G) = n(P)\ord{A}/2\) and \(k(G) = k(P)\ord{A}/2\). The result now follows immediately from \zcref{thm:mainthm}.
\end{proof}

\begin{remark}
For \(G \coloneq \GL(2,q)\), it is known \cite[Tbl.~1]{macd81-a} that \(k(G) = q^2 - 1\). Moreover, since \(\PGL(2,q)\) is the quotient of \(\GL(2,q)\) by its centre (which has order \(q-1\)), the largest nilpotent subgroup of \(\GL(2,q)\) is the preimage of the largest nilpotent subgroup of \(\PGL(2,q)\) under the quotient map. Thus \(n(G) = q^2 - 1 = k(G)\) if and only if \(q \neq 2^r \pm 1\).
\end{remark}

\printbibliography

\end{document}